\documentclass[11pt]{amsart}
\usepackage{amsmath,amsfonts,amssymb,amsthm}
\usepackage{graphics,subfig}
\usepackage{epsfig}
\usepackage{esint}
\usepackage{footmisc}
\usepackage{tikz}

\usepackage{enumerate}
\usepackage{aligned-overset}

\newcommand{ \E }{ \mathbb{E} }

\newcommand{\ignore}[1]{}

\usepackage[hidelinks,hyperindex,breaklinks]{hyperref}

\newcommand{\len}{\text{Len}}

\newtheorem{proposition}{Proposition}
\newtheorem{theorem}{Theorem}
\newtheorem{remark}{Remark}
\newtheorem{lemma}{Lemma}
\newtheorem{corollary}{Corollary}

\title[Min-max optimal matching]{A geometrically linear approximation of min-max optimal matching}

\author{Xiaopeng Cheng, Felix Otto, Matteo Palmieri, Arthur Wachtel}

\begin{document}

\begin{abstract}
In this work we establish a connection between the min-max optimal matching in the critical dimension $d=2$ and a simpler, geometrically linear, action of random curves in a Brownian potential. This is done by following Leighton and Shor and working with the dual formulation, which we approximate by a problem of isoperimetric-type with a random volume term of white-noise character. This allows us to extract the leading-order term in the asymptotics of the expected cost, sharpening previous results.
\end{abstract}
	
\maketitle

\tableofcontents

\section{Introduction}

\subsection{Main result}
The goal of this paper is to draw a connection between two, apparently unrelated, random optimization problems. The first one is two-dimensional and comes from semi-discrete optimal matching: Given the cube $Q_L:=(0,L)^2\subset\mathbb{R}^2$, we consider the restriction $\lambda=dx\llcorner Q_L$ of the two-dimensional Lebesgue measure to $Q_L$, and given a natural number $N$ the atomic measure
\begin{align}\label{M33}
\mu:=\sum_{i=1}^N\delta_{X_i}\quad\mbox{where}\quad \{X_i\}_{i=1}^N~\mbox{are i.~i.~d.~points in $Q_L$ with law $\lambda$.}
\end{align}
With the constraint $N=L^2$, the two measures have the same mass, and we may study the optimal transport problem
\begin{align}\label{a14}
\mathrm{Wass}_\infty(\mu,\lambda):=\inf\{\text{sup}_{(x,y)\in\text{spt}\pi}|x-y|\big|\mbox{coupling $\pi$ between $\lambda$ and $\mu$}\}.
\end{align}

\medskip

On the other side there is a simpler, geometrically linearized, $(1+1)$-dimensional action: Given a two-dimensional white noise $\xi$, an integer $L\ge1$ and a function $h:[0,L]\to\mathbb{R}$, define
\begin{align}\label{M36}
W(h)-D(h):=\int_0^Ldx\int_0^{h(x)}dy~\xi-\int_0^Ldx\frac{1}{2}(\frac{dh}{dx})^2.
\end{align}
On the configuration space
\begin{align}\label{M34}
&\big\{\mbox{continuous functions $h$ $\big|$ $h$ is linear on $[n-1,n]$}\nonumber\\&\qquad\mbox{for $n=1,\ldots,L$ with $h(0)=h(L)=0$}\big\},
\end{align}
consider the variational problem
\begin{align}\label{M38}
A_L:=\max_{h}(W-D)(h)/L.
\end{align}

\medskip

Both problems are critical\footnote{See Section \ref{S:2} for the precise meaning.} and indeed come with two input scales: a macroscopic scale $L$, measuring the domain size; a microscopic scale, given by the typical point-to-point distance in \eqref{M33} and by the spacing between the interpolation points in \eqref{M34}, which is arbitrarily set to 1. A concrete manifestation of criticality is the appearance of logarithmic scalings\footnote{with the notation $A\sim B$ we mean that there exists a universal finite and positive constant $C$ such that $C^{-1}A\le B\le CA$.},
\begin{align}\label{M18}
\E \mathrm{Wass}_\infty(\mu,\lambda)\sim \ln^{3/4}L\quad\mbox{and}\quad\mathbb{E}A_L\sim\ln L,
\end{align}
where the argument of the logarithms should be interpreted as the non-dimensional ratio $L/1$ of the two scales involved. The first result has been known since Leighton-Shor \cite{LS} (see also Talagrand \cite[Chapter 4.7]{T}), while the second one has recently been proved in \cite{OPW}, in the more precise\footnote{The optimal convergence rate has later been found in \cite{OP}. We are actually working in the more natural setting of \cite{OP}, where the ultraviolet cutoff has been imposed only at the level of the configuration space \eqref{M34}, rather than on the field term $W$ in \eqref{M36}.} form:
\begin{align*}
\mbox{there exists}\quad\lim_{L\uparrow\infty}\frac{\mathbb{E}A_L}{\ln L}\in(0,\infty).
\end{align*}

\medskip

Our main result makes this connection rigorous on the level of the leading-order terms of the optimal values in the limit $L\gg1$.

\begin{theorem}\label{T:1}
\begin{align*}
\lim_{L\uparrow\infty}\frac{\mathbb{E}\mathrm{Wass}_\infty(\mu,\lambda)}{\ln^{3/4}L}=\big(\frac{4}{3}\lim_{L\uparrow\infty}\frac{\mathbb{E}A_L}{\ln L}\big)^{3/4}.
\end{align*}
\end{theorem}

\medskip

\subsection{Heuristic derivation}
Let us now explain how to reduce the optimal matching to the maximization of the above action $W-D$. The starting point is the dual formulation of the $\mathrm{Wass}_\infty$ distance, which follows from Strassen's theorem (see \cite[Chapter 10.3, Theorem 8]{Pol}),
\begin{align}\label{m3}
\mathrm{Wass}_\infty(\mu,\lambda)=&\inf\{R\mid \lambda(\Sigma^R)\geq \mu(\Sigma)~\text{for all $\Sigma\subset Q_L$} \}\nonumber\\&
\mbox{where}\quad\Sigma^R:=\{x\in Q_L~|~\text{dist}(x,\Sigma)\le R\}.
\end{align}
This is quite natural: Indeed, let $\pi$ be a plan between $\mu$ and $\lambda$ which moves points by distance at most $R$. Then it transports all the mass $\mu$ contained in a set $\Sigma$ somewhere within $\Sigma^R$, to the effect that $\mu(\Sigma)\le\lambda(\Sigma^R)$. This last inequality must hold for every $\Sigma$, giving a necessary condition for the existence of such a plan $\pi$. Strassen's theorem states that this is also sufficient.

\medskip

Starting from $\eqref{m3}$, we rely on three approximations, which we now explain. The first two are the object of the present work, while the third one has been analyzed in \cite{COP}.

\medskip

\hspace{3mm}\emph{(i) Sharp interface approximation of the area enlargement}. Given a set $\Sigma$, one may rewrite the condition on $R$ in \eqref{m3} in the more suggestive form
\begin{align*}
(\mu-\lambda)(\Sigma)\le\lambda(\Sigma^R)-\lambda(\Sigma),
\end{align*}
where we use the expression ``area enlargement'' to denote the difference on the r.~h.~s.~. For smooth sets, this is classically related to the relative perimeter through the limit 
\begin{align*}
P(\Sigma,Q_L)=\lim_{r\downarrow0}\frac{1}{r}\big(\lambda(\Sigma^r)-\lambda(\Sigma)\big),
\end{align*}
where the l.~h.~s.~denotes the length of the boundary $\partial\Sigma$ contained in $Q_L$. Thus, setting $R:=\mathrm{Wass}_\infty(\mu,\lambda)$ and assuming that in the dual formulation we can restrict our attention to sets $\Sigma$ which are nearly flat at some scale $\gg R$, we have\footnote{with the expression "if $A_1\ll A_2$ then $B_1\approx B_2$", we mean that for every $\epsilon$ there exists $\delta>0$ such that if $A_1\le\delta A_2$, then $(1-\epsilon)B_1\le B_2\le(1+\epsilon)B_1$.}
\begin{align}\label{e8}
\lambda(\Sigma^R)-\lambda(\Sigma)\approx R P(\Sigma,Q_L).
\end{align}
This replacement suggests that
\begin{align}\label{e7}
\mathrm{Wass}_\infty(\mu,\lambda)\approx\max_{\Sigma}\frac{(\mu-\lambda)(\Sigma)}{P(\Sigma,Q_L)}.
\end{align}
The numerator on the r.~h.~s.~is called the discrepancy of $\Sigma$. Such a maximal ratio has already appeared, in a similar fashion, in the work of Leighton and Shor (see Section~\ref{S:2}).
We rigorously prove this first approximation in Subsection~\ref{SS:Linfty} and \ref{SS:dual}.

\medskip

\hspace{3mm}\emph{(ii) From shot noise to white noise.} We observe that, by the central limit theorem, the centered binomial process $\mu-\lambda$ resembles a two-dimensional white noise $\xi$ at mesoscopic scales $1\ll l\ll L$. This clearly does not hold down to scale 1 due to the point process structure, and up to scale $L$ due to the total mass constraint $(\mu-\lambda)(Q_L)=0$. The latter can be fixed by adding a volume compensation, so that
\begin{align}\label{e1}
\mu(\Sigma)-\lambda(\Sigma)\approx \int_\Sigma\xi-\frac{\lambda(\Sigma)}{\lambda(Q_L)}\int_{Q_L}\xi=:W(\Sigma),
\end{align}
provided $\Sigma$ is regular at scale $\gg1$. This white-noise approximation has removed the regularization at scale 1 given by the point-to-point distance. We reintroduce it by restricting the configuration space in the maximization problem to
\begin{align}\label{M24}
\mathcal{P}_{L,1}:=\{\mbox{$\Sigma\subset Q_L$ is a polygon and the edge lengths of $\partial\Sigma\cap Q_L$ are $\ge1$}\}.
\end{align}
Hence, combining \eqref{e7} and \eqref{e1}, we have obtained
\begin{align}\label{M35}
\mathrm{Wass}_\infty(\mu,\lambda)\approx\max_{\Sigma\in\mathcal{P}_{L,1}}\frac{W(\Sigma)}{P(\Sigma,Q_L)}.
\end{align}
In analogy with Leighton and Shor, we call the r.~h.~s.~ Gaussian discrepancy problem.
This second approximation is rigorously proved in Subsections~\ref{S:kmt} and~\ref{SS:smallscale}.

\medskip

\hspace{3mm}\emph{(iii) Geometric linearization of the perimeter.} The aim is to describe the maximizer $\Sigma_*$ of the Gaussian discrepancy problem on the r.~h.~s.~of \eqref{M35}. An easy but fundamental observation is that one can pass from a ratio to an additive formulation, namely
\begin{align*}
&\mbox{$\Sigma_*$ maximizes $W(\Sigma)/P(\Sigma,Q_L)$}\quad\mbox{if and only if}\\&\mbox{$\Sigma_*$ maximizes $W(\Sigma)-I_LP(\Sigma,Q_L)$}\quad\mbox{where}\quad I_L:=\max_{\Sigma\in\mathcal{P}_{L,1}}W(\Sigma)/P(\Sigma,Q_L).
\end{align*}
The second form, rephrased as $\min P-\epsilon W$ with $\epsilon:=I_L^{-1}\ll1$, is at the core of a large-scale regularity theory for $\Sigma_*$, which makes use of the concept of perimeter almost-minimizers and has been studied in \cite{RW}. It tells us that the optimizer $\Sigma_*$ can be locally described, in its own coordinate system, as the subgraph of a function $h_*$. Writing the energy in terms of the latter, we discover that $h_*$ is the maximizer of
\begin{align}\label{M37}
\int dx\int_0^{h(x)}dy~\xi-I_L\int dx\big(\sqrt{1+(\frac{dh}{dx})^2}-1\big).
\end{align}
The geometric linearization, which is a common step in the study of regularity of minimal surfaces, amounts to replacing the second term, the perimeter increment, with the Dirichlet energy:
\begin{align*}
\mbox{if}~|\frac{dh}{dx}|\ll1\quad\mbox{then}\quad\int dx\big(\sqrt{1+(\frac{dh}{dx})^2}-1\big)\approx\int dx\frac{1}{2}(\frac{dh}{dx})^2.
\end{align*}
Substituting into \eqref{M37}, we obtain the action \eqref{M36}.

\medskip

This last geometric linearization has been made rigorous in \cite{COP}, where it has been used to connect the asymptotic behavior of the Gaussian discrepancy problem and the $(1+1)$-dimensional problem. The main result obtained there is the following.

\begin{theorem}[{\cite[Corollary 1]{COP}}]\label{T:4}
\begin{align*}
\lim_{L\uparrow\infty}\frac{1}{\ln^{3/4}L}\E  \sup_{\Sigma\in \mathcal{P}_{L,1}}\frac{W(\Sigma)}{P(\Sigma,Q_L)}=\big(\frac{4}{3}\lim_{L\uparrow\infty}\frac{\mathbb{E}A_L}{\ln L}\big)^{3/4}.
\end{align*}
\end{theorem}

\medskip

\subsection{Extension to bipartite matching}
We conclude the introduction by describing a completely discrete variant of the previous optimal transport problem. Given two sets of independent random variables $(X_i)_{i=1}^N$ and $(Y_i)_{i=1}^N$ with common law $\lambda$, we are interested in the behavior of
\begin{align*}
\min_{\pi}\max_{i=1}^N{|X_i-Y_{\pi(i)}}|,
\end{align*}
for large $N$, where the minimum is taken over permutations $\pi$ of $\{1,...,N\}$. This can be seen as a discrete version of the problem treated before. Indeed, by defining the random measures
\begin{align*}
    \mu_X=\sum_{i=1}^{N}\delta_{X_i} \quad\text{and}\quad\mu_Y=\sum_{i=1}^{N}\delta_{Y_i},
\end{align*}
we recover
\begin{align*}
    \mathrm{Wass}_{\infty}(\mu_X,\mu_Y)= \min_{\pi}\max_{i=1}^N{|X_i-Y_{\pi(i)}}|.
\end{align*}

The techniques developed before allow us to show the following.

\begin{theorem}\label{T:3}
    
    \begin{align*}
        \lim_{L\uparrow\infty}\frac{\E \mathrm{Wass}_{\infty}(\mu_X,\mu_Y)}{\ln^{3/4}L}=\sqrt 2\lim_{L\uparrow\infty}\frac{\E \mathrm{Wass}_{\infty}(\mu,\lambda)}{\ln^{3/4}L}.
    \end{align*}
\end{theorem}

Let us briefly comment on the appearance of the additional $\sqrt 2$-factor: Performing the first approximation, we are left with the study of
\begin{align*}
    \mathrm{Wass}_\infty(\mu_X,\mu_Y)\approx\max_{\Sigma\subset Q_L}\frac{(\mu_X-\mu_Y)(\Sigma)}{P(\Sigma,Q_L)}.
\end{align*}
In the second step, we can then approximate both $\mu_X-\lambda$ and $\mu_Y-\lambda$ by independent white noises $W_X$ and $W_Y$ respectively. Because of independence, $W:=(W_X-W_Y)/\sqrt{2}$ again defines a white noise, amounting to
\begin{align*}
    \max_{\Sigma\in\mathcal{P}_{L,1}}\frac{W_X(\Sigma)-W_Y(\Sigma)}{P(\Sigma,Q_L)}=\sqrt{2}\max_{\Sigma\in\mathcal{P}_{L,1}}\frac{W(\Sigma)}{P(\Sigma,Q_L)}.
\end{align*}

\medskip

\section{Motivation and context}\label{S:2}

The random semi-discrete optimal matching problem \eqref{a14} arose in the computer science literature in connection with the average-case analysis of bin packing algorithms. That bin packing can be recast as a planar matching
problem was already used by Karp, Luby and Marchetti-Spaccamela \cite{KLM}, who introduced the up-right matching and proved preliminary, non-optimal bounds. Motivated by this connection, Leighton and Shor obtained the correct asymptotic behavior \eqref{M18} by analyzing the discrepancy problem %
\begin{align}
\mathrm{Wass}_\infty(\mu,\lambda)\sim\max_{\Sigma}\frac{\mu(\Sigma)-\lambda(\Sigma)}{P(\Sigma)}\quad\mbox{(compare with \eqref{e7})},\label{M19}
\end{align}
where the maximization is taken over lattice animals $\Sigma\subset Q_L$ (cf.~definition of $\mathcal{A}_1$ in Section~\ref{S:kmt}). 
In \cite{T2}, Talagrand embedded the problem into a more general abstract framework, applying the then newly developed machinery of chaining and majorizing measures (see~\cite{T}). 

\medskip

The matching problem, which can be posed in arbitrary dimension $d$, is critical in dimension 2. This can be directly seen from the dual relation in \eqref{M19}. Indeed, the discrepancy of the number of points scales like
\begin{align*}
|\mu(\Sigma)-\lambda(\Sigma)|\sim\sqrt{\lambda(\Sigma)}\sim\text{length}^{d/2},
\end{align*}
for every $\Sigma$ living at a mesoscopic scale $1\ll l\ll L$. This equals the scaling of the surface energy
\begin{align*}
P(\Sigma)\sim\text{length}^{d-1}
\end{align*}
if and only if $d=2$. This results in the problem's criticality, in which case it is common belief that the logarithmic divergences arise from the collective contribution of all scales. A form of equipartition of energy among scales has been rigorously proven for example in \cite[Theorem 2]{OP} for the $(1+1)$-dimensional problem \eqref{M38}.

\medskip

Similar critical behavior has been observed for the $p$-Wasserstein distance\footnote{With this normalization, $\text{Wass}_p(\mu,\lambda)$ is of the order of the typical distance by which the points are moved.}
\begin{align*}
\mathrm{Wass}_p^p(\mu,\lambda)=\inf\big\{\frac{1}{L^2}\int |x-y|^p\pi(dx,dy)~\big|~\mbox{coupling $\pi$ between $\lambda$ and $\mu$}\big\},
\end{align*}
for which in \cite{AKT} it has been shown that
\begin{align}\label{akt}
\mathbb{E}^{1/p}\mathrm{Wass}^p_p(\mu,\lambda)\sim_p\ln^{1/2}L\quad\mbox{for all}~1\le p<\infty.
\end{align}
The particular case $p=2$ is by far the best understood: In \cite{CLPS}, the authors proposed a heuristic based on a formal linearization of the Monge-Amp\`ere equation to the Poisson equation, which allowed to identify the explicit constant in the leading-order term of the expansion
\begin{align*}
\mathbb{E}^{1/2}\mathrm{Wass}_2^2(\mu,\lambda)\approx\textstyle\frac{1}{\sqrt{4\pi}}\ln^{1/2}L.
\end{align*}
This prediction has successfully been made rigorous in \cite[Theorem 1.1]{AST}, and an almost-optimal error estimate has been obtained in \cite[Theorem 1]{GHO}. 

\medskip

For the case $p\in(1,\infty),p\neq2$, the problem has recently been reduced to understanding the solution of the $p'$-Laplace equation with white-noise forcing (see~\cite{ATV}) and the existence of a limiting constant has been announced recently by Armegioiu, Goldman, Grotto and Trevisan. For this, one may start from the Benamou-Brenier formula,
\begin{align*}
\frac{1}{p}\mathrm{Wass}_p^p(\mu,\lambda)&=\inf\big\{\frac{1}{L^2}\int_0^1\int\frac{1}{p}|\frac{j_t}{\rho_t}|^{p} d\rho_t\mid \partial_t\rho_t+\mbox{div}j_t=0,\rho_0=\mu,\rho_1=\lambda\big\}.
\end{align*}
In the regime $\mu\approx \lambda$, we may approximate $\rho_t\approx 1$ and perform the following geometric linearization: Defining the time-averaged flux $\bar j:=\int_{0}^{1}j_t$, we have
\begin{align*}
\frac{1}{p}\mathrm{Wass}_p^p(\mu,\lambda)\approx \inf\big\{\frac{1}{L^2}\int\frac{1}{p}|\bar j|^p \mid \mbox{div}\bar j=\mu-\lambda \big\}.
\end{align*}
Passing to the dual of this convex minimization problem, we obtain
\begin{align*}
\frac{1}{p}\mathrm{Wass}_p^p(\mu,\lambda)\approx\sup\big\{\frac{1}{L^2}\int\phi(\mu-\lambda)-\frac{1}{p'}|\nabla\phi|^{p'}\big\}.
\end{align*}
so that one can see the appearance of the $p'$-Laplacian on the r.~h.~s.~. Now replace $\phi$ by $t\phi$  and optimize in $t$ to obtain
\begin{align*}
\frac{1}{p}\mathrm{Wass}_p^p(\mu,\lambda)\approx\frac{1}{p}\sup\{ \frac{1}{L^2}|\frac{\int \phi (\mu-\lambda)}{(\int|\nabla\phi|^{p'})^{1/p'}}|^p\}
\end{align*}
From this last expression, one may take the formal limit $p\to\infty$; in view of the $1$-homogeneity, it is sufficient to consider indicator functions $\phi=I(\Sigma)$, thus recovering the discrepancy problem \eqref{M19}. With this point of view, Proposition~\ref{P:1} below is the analog of \cite[Theorem 1.2]{ATV}.

\medskip

Morally, our strategy resembles the one in \cite{AST}, in the sense that both approaches linearize the dual formulation of the Wasserstein distance.
However, \cite{AST} relies on explicit gradient estimates for the Green function of the Poisson equation, while this work proves existence of the limit by approximating the random forcing by a Gaussian version of it, which has recently been treated in \cite{COP}. The main tool of this approximation is the celebrated KMT-coupling established in \cite{KMT}, a streamlined exposition of this rather delicate construction can be found in \cite[Chapter 10]{Pol}. This method is not new in the matching literature, as it has already been employed in the aforementioned work by Ajtai, Koml\'os and Tusn\'ady \cite{AKT} to prove upper bounds for the transport distance $\mathrm{Wass}_p(\mu,\lambda)$. In contrast, we are in need of an approximation which is uniform over the class $\mathcal{P}_{L,1}$. While this does not fall within the scope of the multivariate constructions of \cite{Mas}, our method is somewhat reminiscent of his approach.

\section{Proof}

\subsection{Reduction to \texorpdfstring{$L^\infty$ }{} densities}\label{SS:Linfty}
The identification of the dual problem as in \eqref{e7} is simpler when  $\mu$ is a measure with $L^\infty$ density. Hence, as a first step we show that we may replace the singular measure $\mu$ by a measure which is close to $\lambda$ in $L^\infty$. This is done by coarse-graining $\mu$. Given a mesoscopic scale $1\le l\le L$ such that $L/l\in\mathbb{N}$, we partition (up to $\lambda$-negligible sets) $Q_L=\cup_{i} Q_{l,i}$ into $(L/l)^2$ subcubes of size $l$ and define
\begin{align}\label{M20}
\mu_l=\rho_l\lambda:=\sum_{i=1}^{(L/l)^2}\frac{\mu(Q_{l,i})}{\lambda(Q_{l,i})}\lambda\llcorner Q_{l,i}.
\end{align}
We note that $\mu_l$ has the same mass as $\lambda$ (with probability 1). Moreover
\begin{align*}
&\mbox{$\mu(Q_{l,i})=\mu_l(Q_{l,i})$ for every $i$}\quad\mbox{to the effect that}\quad \mathrm{Wass}_\infty(\mu,\mu_l)\le\sqrt{2}l,
\end{align*}
where the last inequality follows by taking any coupling $\pi$ of $\mu$ and $\mu_l$ supported on $\cup_i Q_{l,i}\times Q_{l,i}$. By the triangle inequality, we have
\begin{align}\label{M2}
|\mathrm{Wass}_\infty(\mu,\lambda)-\mathrm{Wass}_\infty(\mu_l,\lambda)|\le \mathrm{Wass}_\infty(\mu,\mu_l)\le\sqrt{2}l,
\end{align}
thus it suffices to study the asymptotics of $\mathrm{Wass}_\infty(\mu_l,\lambda)$ provided we choose
\begin{align}\label{M1}
l\ll\ln^{3/4}L.
\end{align}

\medskip

Next we show that $\mu_l$ and $\lambda$ are close in $L^\infty$ with overwhelming probability as long as
\begin{align}\label{M3}
l\gg\ln^{1/2}L.
\end{align}
This is the main advantage in passing from $\mu$ to its coarse-grained version $\mu_l$. Note that this is compatible with the condition \eqref{M1}. On the one hand this would not be true for $p<\infty$, since \eqref{M1} should be replaced by $l\ll\ln^{1/2}L$ (see~\eqref{akt}), making $p=\infty$ easier; on the other hand, for finite $p$ the estimate \eqref{M2} is clearly suboptimal, in view of averaging effects. However, for $p=\infty$, \eqref{M2} is essentially optimal, in view of the fact that the $\mathrm{Wass}_\infty$-distance may be influenced by sporadic small-scale defects. The existence of some room between the regimes \eqref{M1} and \eqref{M3} has in fact to be expected, since the problem is critical and as such should not be influenced by the small-scale behavior.

\begin{lemma}\label{L:3}
If $l\gg\ln^{1/2}L$ and $\nu\gg1$, then\footnote{a statement of the form ''if $A_1\ll A_2$ then $B_1\lesssim B_2$'' means that there exist two universal constants $\varepsilon>0,C<+\infty$ such that if $A_1\le\varepsilon A_2$ then $B_1\le CB_2$.} 
\begin{align*}
\ln\mathbb{P}(\|\rho_l-1\|_{L^\infty}\ge\nu l^{-1}\ln^{1/2}L)\lesssim -\nu\ln L.
\end{align*}
\end{lemma}

{\sc Proof.} This follows from the Bernstein inequality \cite[Theorem 2.10]{BLM}, applied to the binomials $\mu(Q_{l,i})$ of parameters $n=L^2$ and $p=(l/L)^2$. Appealing to the union bound, we estimate for $\nu\gg1$
\begin{align*}
    \mathbb{P}(\|\rho_l-1\|_{L^\infty}\ge\nu l^{-1}\ln^{1/2}L)&\overset{\eqref{M20}}{\leq} \sum_{i=1}^{(L/l)^2}\mathbb{P}(|\mu(Q_{l,i})-l^2|\ge\nu l\ln^{1/2}L)\\
    &\leq (\frac L l)^2\exp(-\frac{\nu ^2l^2\ln L}{2l^2+\nu l\ln^{1/2}L})\\
    &\leq\exp(-\frac{\nu^2l\ln L}{\nu(l+\ln^{1/2}L)}+2\ln L).
\end{align*}
Since we can absorb $\ln^{1/2}L$ into $l$ in the denominator, we conclude.
\qed

\medskip

\subsection{The dual problem}\label{SS:dual}
In this section, we justify the approximation \eqref{e7} for $\mathrm{Wass}_\infty(\rho_l\lambda,\lambda)$.
The main result is the following.

\begin{proposition}\label{P:1}
Let $\rho$ be a density with $\int\rho d\lambda=\int 1d\lambda$. If $\|\rho-1\|_{L^\infty}\ll1$, 
\begin{align*}
\mathrm{Wass}_\infty(\rho\lambda,\lambda)=\big(1+O(\|\rho-1\|_{L^\infty})\big)\max_{\Sigma\subset Q_L}\frac{(\rho-1)\lambda(\Sigma)}{P(\Sigma,Q_L)}.
\end{align*}
\end{proposition}

The proof mainly consists of making the approximation \eqref{e8} rigorous. We argue differently in the upper and lower bounds.

\begin{lemma}\label{L:1}
\begin{align*}
\mathrm{Wass}_\infty(\rho\lambda,\lambda)\le(1-\|\rho-1\|_{L^\infty})^{-1}\max_{\Sigma\subset Q_L}\frac{(\rho-1)\lambda(\Sigma)}{P(\Sigma,Q_L)}.
\end{align*}
\end{lemma}

{\sc Proof.} Let us denote the r.~h.~s.~by $R$; by the dual formulation \eqref{m3}, we need to show for all $\Sigma\subset Q_L$, that
\begin{align}\label{M21}
\lambda(\Sigma^R)-\rho\lambda(\Sigma)\ge0.
\end{align}
Now let $d(\cdot,\Sigma)$ be the distance function to a closed $\Sigma\subset Q_L$, then $|\nabla d(\cdot,\Sigma)|=I(\Sigma^c)(\cdot)$ as $L^1$ functions  (see~\cite[Theorem 3.14]{EG}). Applying coarea formula to $d(\cdot,\Sigma)$ gives
\begin{align*}
\lambda(\Sigma^R)-\lambda(\Sigma)&=\int_{Q_L\cap \Sigma^R}|\nabla \text{dist}(x,\Sigma)|dx\\&=\int_{0}^RP(\{\text{dist}(x,\Sigma)>r\},Q_L\cap\Sigma^R)dr=\int_0^RP(\Sigma^r,Q_L)dr,
\end{align*}
which implies the existence of a radius $r\in(0,R)$ such that
\begin{align}\label{m2}
\lambda(\Sigma^R)-\lambda(\Sigma)\ge RP(\Sigma^{r},Q_L).
\end{align}

\medskip

We now turn to the l.~h.~s.~ of \eqref{M21}:
\begin{align*}
&\lambda(\Sigma^R)-\rho\lambda(\Sigma)=\big(\lambda(\Sigma^R)-\lambda(\Sigma)\big)-(\rho-1)\lambda(\Sigma)\\
&=\lambda(\Sigma^R)-\lambda(\Sigma)+(\rho-1)\lambda(\Sigma^r\backslash\Sigma)-(\rho-1)\lambda(\Sigma^r)\\
&\ge (\lambda(\Sigma^R)-\lambda(\Sigma))(1-\|\rho-1\|_{L^\infty})-(\rho-1)\lambda(\Sigma^r)&&\mbox{(since $\Sigma^r\subset\Sigma^R$)}\\
&\ge RP(\Sigma^r,Q_L)(1-\|\rho-1\|_{L^\infty})-(\rho-1)\lambda(\Sigma^r)&&\mbox{(using \eqref{m2})}.
\end{align*}
Since the last term is nonnegative by definition of $R$, we conclude.
\qed

\begin{lemma}\label{L:2}
\begin{align*}
\mathrm{Wass}_\infty(\rho\lambda,\lambda)\ge(1+\|\rho-1\|_{L^\infty})^{-1}\max_{\Sigma\subset Q_L}\frac{(\rho-1)\lambda(\Sigma)}{P(\Sigma,Q_L)}.
\end{align*}
\end{lemma}

{\sc Proof.} The strategy consists of considering an optimal set
\begin{align}\label{m4}
\Sigma_*:=\arg\max_{\Sigma\subset Q_L}\frac{(\rho-1)\lambda(\Sigma)}{P(\Sigma,Q_L)}\quad\mbox{with ratio}\quad R_*:=\frac{(\rho-1)\lambda(\Sigma_*)}{P(\Sigma_*,Q_L)}
\end{align}
and using it as a competitor in the dual formulation \eqref{m3}. We will show below a regularity result of $\Sigma_*$, which will allow us to bound the area enlargement by the perimeter, namely for every $R\ge 0$
\begin{align}\label{m5}
\lambda(\Sigma_*^R)-\lambda(\Sigma_*)\le(1+\kappa R)RP(\Sigma_*,Q_L)\quad\mbox{with}\quad\kappa:=\|\rho-1\|_{L^\infty}/R_*.
\end{align}
This suffices to conclude: In view of the dual formulation, in order to show that $\mathrm{Wass}_\infty(\rho\lambda,\lambda)\ge R_*(1+\|\rho-1\|_{L^\infty})^{-1}$, we need to prove
\begin{align*}
\lambda(\Sigma_*^R)<\rho\lambda(\Sigma_*)\quad\mbox{if}\quad R<R_*(1+\|\rho-1\|_{L^\infty})^{-1},
\end{align*}
which follows from \eqref{m5} and the definition of $R_*$ in \eqref{m4},
\begin{align*}
\lambda(\Sigma_*^R)-\rho\lambda(\Sigma_*)&=\big(\lambda(\Sigma_*^R)-\lambda(\Sigma_*)\big)-(\rho-1)\lambda(\Sigma_*)\\
&\le(1+\kappa R)RP(\Sigma_*,Q_L)-R_*P(\Sigma_*,Q_L)\\&\le(1+\|\rho-1\|_{L^\infty}-R_*/R)RP(\Sigma_*,Q_L).
\end{align*}

\medskip

Here comes the proof of \eqref{m5}. We rephrase the optimality condition \eqref{m4} as
\begin{align*}
&0=P(\Sigma_*,Q_L)-R_*^{-1}(\rho-1)\lambda(\Sigma_*)\le P(\Sigma,Q_L)-R_*^{-1}(\rho-1)\lambda(\Sigma),
\end{align*}
which implies
\begin{align}\label{M6}
P(\Sigma_*,Q_L)\le P(\Sigma,Q_L)+\kappa\lambda(\Sigma_*\Delta\Sigma)\quad\mbox{with}\quad\kappa:=\|\rho-1\|_{L^\infty}/R_*.
\end{align}
A set $\Sigma_*$ satisfying such an inequality for all $\Sigma$ is said to be a $\kappa$-perimeter minimizer. It is a classical result (cf.~Lemma~\ref{L:reg} in the appendix) that
\begin{align}
&\mbox{the connected components of $\partial\Sigma_*\cap Q_L$ are $C^{1,1}$ curves}\nonumber\\&\mbox{with modulus of the curvature $\le\kappa$.}\label{m7}
\end{align}

\medskip

Let us now show \eqref{m5} assuming that $\partial\Sigma_*\cap Q_L$ consists of only one curve and let $\gamma(s)$ be its arc-length parameterization. The case of multiple curves is analogous. Because of \eqref{m7}, the map
\begin{align*}
f:[0,P(\Sigma_*,Q_L)]\times[0,R]\to\mathbb{R}^2\quad\mbox{given by}\quad f(s,t):=\gamma(s)+t\gamma'^{\perp}(s),
\end{align*}
is Lipschitz and its image contains $\Sigma_*^R\backslash\Sigma_*$, so that by the area formula
\begin{align}\label{a15}
\lambda(\Sigma_*^R)-\lambda(\Sigma_*)\le\int_0^Rdt\int_0^{P(\Sigma_*,Q_L)}ds|\det\nabla f(s,t)|.
\end{align}
A short computation gives
\begin{align}\label{m14}
|\det\nabla f(s,t)|=|1+t\gamma''^\perp\cdot\gamma'|\le1+\kappa R,
\end{align}
so that \eqref{m5} follows. \qed

\subsection{Gaussian approximation: KMT coupling}\label{S:kmt}
In this and the next subsection, we justify the Gaussian approximation \eqref{M35} for the discrepancy problem identified in the last subsection.  We will quantify the size of random errors in terms of Orlicz norms: For $s\in[1,\infty)$ and a random variable $X$,
\begin{align*}
\|X\|_s:=\inf\{N>0~|~\mathbb{E}\exp(\frac{|X|}{N})^s\le e\}.
\end{align*}
As a first step, we approximate $(\rho_l-1)\lambda$ by a piecewise constant Gaussian field, defined as follows: Let $\xi$ be a two-dimensional white noise and let $\xi_l\in L^\infty(Q_L)$ be its coarse-grained approximation at scale $l$,
\begin{align}\label{M25}
\xi_l:=\sum_{i=1}^{(L/l)^2}\frac{I(Q_{l,i})}{\lambda(Q_{l,i})}\int_{Q_{l,i}}\xi.
\end{align}
We observe that this is analogous to the way $\rho_l$ is obtained from $\mu$. With this at hand, we define a Gaussian field term for $\Sigma\subset\mathbb{R}^2$ by\footnote{Here $I(Q)$ denotes the indicator function of $Q$.}
\begin{align}
W_l(\Sigma):=(\xi_l\lambda)(\Sigma)-\frac{\lambda(\Sigma)}{\lambda(Q_L)}(\xi_l\lambda)(Q_L)\label{M4}.
\end{align}

The goal of this subsection is to prove the following.
\begin{lemma}\label{L:4}
There exists a coupling between $\xi$ and $\mu$ with
\begin{align*}
\big\|\max_{\Sigma\subset Q_L}\frac{(\rho_l-1)\lambda(\Sigma)}{P(\Sigma,Q_L)}-\max_{\Sigma\subset Q_L}\frac{W_l(\Sigma)}{P(\Sigma,Q_L)}\big\|_1\lesssim l^{-1}\ln L.
\end{align*}
\end{lemma}

We note that, since both terms are expected to scale like $\ln^{3/4}L$, this is a good approximation provided that $l\gg\ln^{1/4}L$, which is a weaker assumption than the one in \eqref{M3}.

\medskip

The main tool used in the proof is a result by Koml\'os-Major-Tusn\'ady, which constructs such a coupling uniformly over a collection of rectangles. Similarly to \cite{Mas}, we will employ the one-dimensional version, which we then lift to the two-dimensional statement. For the reader's convenience, we recall the precise form of the KMT-coupling we will use.

\begin{theorem}[{Koml\'os-Major-Tusn\'ady \cite{KMT}}]\label{KMT}
Let $\xi^{1d}$ be a white noise on the interval $[0,1]$. There exists a constant $C$ such that for each $n\in\mathbb{N}$ there exist i.i.d. $Z_1,...,Z_n$ with uniform distribution on the interval $[0,1]$ satisfying %
\begin{align*}
\ln\mathbb{P}\big(\sup_{t\leq 1}\big|\mu^{1d}([0,t])-nt-\sqrt n\big(\xi^{1d}([0,t])-t\xi^{1d}([0,1])\big)\big|\geq\nu+C\ln n\big)\lesssim -\nu,
\end{align*}
with $\mu^{1d}=\sum_{i=1}^n\delta_{Z_i}$.

\end{theorem}

\begin{remark}\label{rk:5}
In the original work, the statement is written in an integrated form with a Brownian bridge $(B(t))_{t\in[0,1]}=(\xi^{1d}([0,t])-t\xi^{1d}([0,1]))_{t\in[0,1]}$ and empirical distribution function $F_n(t)=\frac 1n\sum_{i=1}^nI(Z_i\leq t)=\frac1n\mu^{1d}([0,t])$:
\begin{align*}
\ln\mathbb{P}(\sup_{t\leq 1}|n(F_n(t)-t)-\sqrt{n}B(t)|\ge\nu+C\ln n)\lesssim -\nu.
\end{align*}
\end{remark}
A simple post-processing of the KMT theorem is the following corollary, the proof of which can be found in the appendix.

\begin{corollary}\label{L:7}
    Let $T>0$ and $\xi^{1d}$ a white noise on the interval $[0,T]$. For each $n\in\mathbb{N}$ there exists a binomial process $\mu^{1d}=\sum_{i=1}^n\delta_{Z_i}$ such that  
    \begin{align*}
        &\ln\mathbb P\big(\sup_{I\subset[0,T]}|\mu^{1d}(I)-\frac{n}{T}|I|-\sqrt{\frac nT}\big(\xi^{1d}(I)-\frac{|I|}{T}\xi^{1d}([0,T])\big)|\ge\nu+C\ln n\big)\\&\lesssim-\nu,
    \end{align*}
    where the supremum is taken over all intervals $I$.
\end{corollary}

In order to make use of the KMT-coupling in Lemma \ref{L:4}, we need to pass from one dimensional to two dimensional sets first. We start by extending the coupling to a specific class of sets $\Sigma\subset Q_L$ that are unions of some of the squares $\{Q_{l,i}\}_i$. Such sets are called \emph{lattice animals}, and we denote their family by $\mathcal{A}_l$. The result in Lemma~\ref{L:4} for all sets will be recovered from the following by approximation.

\medskip

\begin{lemma}\label{L:5}
There exists a coupling between $\xi$ and $\mu$ such that
\begin{align*}
\|\max_{\Sigma\in\mathcal{A}_l}\frac{|(\rho_l-1)\lambda(\Sigma)-W_l(\Sigma)|}{P(\Sigma,Q_L)}\|_1\lesssim l^{-1}\ln L.
\end{align*}
\end{lemma}

{\sc Proof of Lemma \ref{L:5}.} 
The proof consists of two steps: We first map each lattice animal to a union of one-dimensional sets contained in the interval $[0,L^2/l]$. By applying the rescaled version of the KMT-coupling (Corollary \ref{L:7}) we then deduce that there exist a point process and a Gaussian process which are close in a quantitative sense. Inverting the map from above will then allow us to deduce the statement.
\medskip

We begin by introducing the (up to a set of measure zero) bijective map $\tau:[0,L^2/l]\times[0,l]\to \overline{Q_L}$,
\begin{align}\label{a2}
\tau(x,y):=(x-L\bigl\lfloor\frac{x}{L}\bigr\rfloor,y+l\bigl\lfloor\frac xL\bigr\rfloor).
\end{align}
Informally, $\tau$ stacks the $L/l$ lexicographically ordered rectangles of the form $L[j,j+1]\times[0,l]$ on top of each other. One particular feature is that for any $\Sigma\in\mathcal{A}_l$, $\overline{\tau^{-1}(\Sigma)}$ is a union of disjoint rectangles of the form $I\times[0,l]$ with $I\subset [0,L^2/l]$ being a closed interval. Such rectangles will be called components of $\Sigma$ and we will call this collection $\mathrm{Comp}(\Sigma)$. 

\medskip

With this convention, we claim\footnote{Here $|\cdot|$ denotes the cardinality of a set.} for any $\Sigma\neq Q_L$,
\begin{align}\label{a1}
    P(\Sigma,Q_L)\geq l|\mathrm{Comp}(\Sigma)|.
\end{align}
Indeed, let $j,k\in\mathbb{Z}$ and $C=[jl,kl]\times[0,l]\subset[0,\frac{L^2}l]\times[0,l]$ be a component of $\Sigma$ and introduce the cube to left of $C$, $Q_{\text{left}}:=[(j-1)l,jl]\times[0,l]$, and to the right $Q_{\text{right}}:=[kl,(k+1)l]\times[0,l]$. Since $\Sigma\neq Q_L$, at least one of these has to be contained in the strip $[0,\frac{L^2}l]\times[0,l]$. Moreover, by definition of components, both cubes $\tau(Q_{\text{left}})$ and $\tau(Q_{\text{right}})$ are not contained in $\Sigma$. The claim will now follow from the fact that at least one of them shares an edge with $\Sigma$, that is not contained in $\partial Q_L$, thus contributing length $l$ to the relative perimeter $P(\Sigma,Q_L)$. To see that one of the cubes is adjacent to $\Sigma$, we distinguish between two cases:
\begin{enumerate}
    \item The component is actually of the form $C=[jL,kL]\times[0,l]$, i.e. the vertical part of the boundary of $\tau(C)$ belongs to $\partial Q_L$. Choosing out of $Q_{\text{left}}$ and $Q_{\text{right}}$ the cube contained in the strip, and calling it $Q$, we realize that $\tau(Q)$ shares a horizontal edge with $\tau(C)\subset\Sigma$.
    \item If $C$ is not of the form above, then either $\tau(Q_\text{left})$ or $\tau(Q_\text{right})$ share a vertical edge with $\tau(C)\subset\Sigma$.
\end{enumerate}

\medskip

The main advantage of \eqref{a1} is that it allows us to control the size of the class of sets we are approximating over. Decomposing the lattice animal $\Sigma$ into its components,
\begin{align}
\frac{|(\rho_l-1)\lambda(\Sigma)-W_l(\Sigma)|}{P(\Sigma,Q_L)}&\leq \sum_{C\in\mathrm{Comp}(\Sigma)}\frac{|(\rho_l-1)\lambda(\tau(C))-W_l(\tau(C))|}{P(\Sigma,Q_L)}\nonumber\\
\overset{\eqref{a1}}&{\leq} \sup_{C\in\mathrm{Comp}(\Sigma)}\frac{|(\rho_l-1)\lambda(\tau(C))-W_l(\tau(C))|}{l}\nonumber\\
&\leq \sup_{I\subset[0,\frac{L^2}{l}]}\frac{|(\rho_l-1)\lambda(\tau(I\times[0,l]))-W_l(\tau(I\times[0,l]))|}{l}.\label{a4}
\end{align}
We can interpret $\mathrm{Comp}(\Sigma)$ as a collection of closed intervals contained in $[0,L^2/l]$, by applying the projection onto the first coordinate $(x,y)\mapsto x$, and therefore we are left with approximating the point process by a white noise over a one-dimensional, low-complexity index set.

\medskip

Starting from $\xi$, since $\tau$ is measure-preserving, we have that
\begin{align}\label{a8}
\sqrt{l}\,\xi^{1d}(\cdot):=\xi\circ\tau(\cdot\times[0,l])\quad\mbox{is a white noise on $[0,L^2/l]$.}
\end{align}
Appealing to Corollary \ref{L:7} with $n=L^2$ and $T=\frac{L^2}{l}$, we see that there exists a one-dimensional binomial process 
\begin{align*}
\mu^{1d}=\sum_{i=1}^{L^2}\delta_{Z_i}, \quad Z_1,...,Z_{L^2} \text{  i.~i.~d.~ uniform on  }[0,L^2/l], 
\end{align*}
such that, by Lemma \ref{L:11},
\begin{align}\label{a5}
\big\|\sup_{I\subset[0,L^2/l]} |\mu^{1d}(I)-l|I|-\sqrt l\big(\xi^{1d}(I)-\frac{l|I|}{L^2}\xi^{1d}([0,\frac{L^2}{l}])\big)|\big\|_1\lesssim \ln L.
\end{align}

Using again that $\tau$ is measure-preserving, we can construct $\mu$ such that almost surely,
\begin{align}\label{a7}
\mu^{1d}(I)=\mu\circ\tau(I\times[0,l])\quad\mbox{for every $I$}, 
\end{align}
by independently assigning each point a random height $h_i$, which is uniformly distributed on $[0,l]$. We can then recover $\mu$ on $Q_L$ by applying $\tau^{-1}$:
\begin{align*}
    \mu=\sum_{i=1}^n\delta_{\tau(Z_i,h_i)}.
\end{align*}
\medskip

Let us now conclude. Since the r.~h.~s.~ in \eqref{a4} is independent of $\Sigma$, writing it in terms of the one-dimensional objects $\mu^{1d}$ and $\xi^{1d}$ yields
\begin{align*}
&\max_{\Sigma\in\mathcal{A}_l}\frac{|(\rho_l-1)\lambda(\Sigma)-W_l(\Sigma)|}{P(\Sigma,Q_L)}\\&\leq \sup_{I\subset \big[0,\frac{L^2}l\big]}\frac{|\mu^{1d}(I)-l|I|-\sqrt{l}\big(\xi^{1d}(I)-\frac{l|I|}{L^2}\xi^{1d}([0,\frac{L^2}{l}])\big)|}{l}.
\end{align*}
Using \eqref{a5}, we conclude.
\qed

\medskip

{\sc Proof of Lemma~\ref{L:4}.} The goal is to show that
\begin{align}\label{M12}
\max_{\Sigma\subset Q_L}\frac{(\rho_l-1)\lambda(\Sigma)}{P(\Sigma,Q_L)}-\max_{\Sigma\subset Q_L}\frac{W_l(\Sigma)}{P(\Sigma,Q_L)}\lesssim\max_{\Sigma\in\mathcal{A}_l}\frac{|(\rho_l-1)\lambda(\Sigma)-W_l(\Sigma)|}{P(\Sigma,Q_L)}.
\end{align}
Since an analogous estimate holds after exchanging the two terms on the l.~h.~s., the proof follows by Lemma \ref{L:5}.

\medskip

Let $\Sigma_*$ be the optimal set for the first maximization problem in \eqref{M12}. Using it as a competitor for the Gaussian problem, we have that the l.~h.~s.~of \eqref{M12} is bounded by
\begin{align*}
\frac{|(\rho_l-1)\lambda(\Sigma_*)-W_l(\Sigma_*)|}{P(\Sigma_*,Q_L)},
\end{align*}
so that it remains to show that the above is estimated by the r.~h.~s.~of \eqref{M12}. 

\medskip

For every set $\Sigma$, we introduce the inner approximation by lattice animals
\begin{align}\label{M16}
\mbox{$\mathring{\Sigma}:=$ union of cubes $Q_{l,i}$ such that $Q_{l,i}\subset\Sigma$.}
\end{align}
We will prove in the sequel that $\Sigma_*$ is regular enough to satisfy
\begin{align}\label{M8}
&P(\mathring{\Sigma}_*,Q_L)\lesssim P(\Sigma_*,Q_L)\quad\mbox{and}\quad |\mathcal{I}|\lesssim P(\Sigma_*,Q_L)/l\nonumber\\
&\mbox{where}\quad\mathcal{I}:=\{i~|~Q_{l,i}\cap\Sigma_*\neq\emptyset~\mbox{and}~Q_{l,i}\not\subset\Sigma_*\}.
\end{align}

\medskip

With this at hand, we can conclude. We apply the triangle inequality,
\begin{align}
|(\rho_l-1)\lambda(\Sigma_*)-W_l(\Sigma_*)|&\le|(\rho_l-1)\lambda(\mathring{\Sigma}_*)-W_l(\mathring{\Sigma}_*)|\nonumber\\
&+\sum_{i\in\mathcal{I}}|(\rho_l-1)\lambda(\Sigma_*\cap Q_{l,i})-W_l(\Sigma_*\cap Q_{l,i})|.\label{M9}
\end{align}
Introducing $\mathcal{C}$ as the r.~h.~s.~of \eqref{M12}, we have
\begin{align}
&|(\rho_l-1)\lambda(\mathring{\Sigma}_*)-W_l(\mathring{\Sigma}_*)|\lesssim\mathcal{C}P(\mathring{\Sigma}_*,Q_L)\lesssim\mathcal{C}P(\Sigma_*,Q_L)\label{M13}
\end{align}
where the last inequality follows from \eqref{M8}. With respect to the second term in \eqref{M9}, since $(\rho_l-1)\lambda-W_l$ has constant density inside each cube $Q_{l,i}$, 
\begin{align*}
|(\rho_l-1)\lambda(\Sigma_*\cap Q_{l,i})-W_l(\Sigma_*\cap Q_{l,i})|\le|(\rho_l-1)\lambda(Q_{l,i})-W_l(Q_{l,i})|.
\end{align*}
Moreover, using that $Q_{l,i}\in\mathcal{A}_l$ and the second item in \eqref{M8},
\begin{align}
\sum_{i\in\mathcal{I}}|(\rho_l-1)\lambda(Q_{l,i})-W_l(Q_{l,i})|\lesssim\mathcal{C}l|\mathcal{I}|\lesssim\mathcal{C}P(\Sigma_*,Q_L).\label{M14}
\end{align}
Plugging the estimates \eqref{M13} and \eqref{M14} into \eqref{M9}, we deduce \eqref{M12}.

\medskip

Here comes the argument for \eqref{M8}. Let us first prove that
\begin{align}\label{M7}
\mbox{$\partial\Sigma_*\cap Q_L$ is the union of $C^{1,1}$ curves with curvature $\lesssim l^{-1}$.}
\end{align}
Indeed, from \eqref{M6} and \eqref{m7}, we already know that the curvature is bounded by $\|\rho_l-1\|_{L^\infty}/R_*$, hence it is enough to show that $R_*\gtrsim l\|\rho_l-1\|_{L^\infty}$. This follows directly from using the squares $Q_{l,i}$ and their complements $Q_L\backslash Q_{l,i}$ as competitors. Indeed, since $\rho_l-1$ is constant over every $Q_{l,i}$,
\begin{align*}
R_*\ge&\max\big\{\max_i\frac{(\rho_l-1)\lambda(Q_{l,i})}{P(Q_{l,i},Q_L)},\max_i\frac{(\rho_l-1)\lambda(Q_L\backslash Q_{l,i})}{P(Q_L\backslash Q_{l,i},Q_L)}\big\}\\
=&\max\big\{\max_i\frac{(\rho_l-1)\lambda(Q_{l,i})}{P(Q_{l,i},Q_L)},\max_i-\frac{(\rho_l-1)\lambda(Q_{l,i})}{P(Q_{l,i},Q_L)}\big\}\sim l\|\rho_l-1\|_{L^\infty}.
\end{align*}

\medskip

Let us finally show that the regularity in \eqref{M7} implies \eqref{M8}. Note that
\begin{align*}
P(\mathring{\Sigma}_*,Q_L)/l\lesssim|\{i~|~Q_{l,i}~\mbox{is adjacent to}~\bar\Sigma_*\}|=:|\mathcal{J}|.
\end{align*}
We observe that if $i\in\mathcal{J}\cup\mathcal{I}$, then $Q_{l,i}$ is contained in the $\sqrt{2}l$-neighborhood of the boundary $\partial\Sigma_*$, which can be written as
\begin{align*}
(\partial\Sigma_*)^{\sqrt{2}l}=\big(\Sigma_*^{\sqrt{2}l}\backslash\Sigma_*\big)\cup\big((\Sigma_*^c)^{\sqrt{2}l}\backslash\Sigma_*^c\big)\quad\mbox{with}\quad\Sigma^c_*:=Q_L\backslash\Sigma_*.
\end{align*}
Therefore, applying the estimate on the area enlargement in \eqref{m5},
\begin{align*}
|\mathcal{J}\cup\mathcal{I}|\le\lambda\big((\partial\Sigma_*)^{\sqrt{2}l}\big)/l^2\lesssim(1+\kappa l) P(\Sigma_*,Q_L)/l.
\end{align*}
The claim \eqref{M8} then follows by using the bound on the curvature in \eqref{M7}.
\qed

\medskip

\subsection{Change of small-scale cutoff}\label{SS:smallscale}
We are left with the study of 
\begin{align*}
\max_{\Sigma\subset Q_L}\frac{W_l(\Sigma)}{P(\Sigma,Q_L)}.
\end{align*}
We note that by the scaling $x=l\hat x$, we can assume $l=1$. We still keep $L$ as the macroscopic scale instead of $L/l$, since the difference is negligible on the logarithmic scale
\begin{align*}
\ln^{3/4}(L/l)\overset{\eqref{M1}}{=}\ln^{3/4}L+O(\ln\ln L/\ln^{1/4}L).
\end{align*}
Therefore, by Theorem \ref{T:4}, it suffices to show 
\begin{align*}
\max_{\Sigma\subset Q_L}\frac{W_1(\Sigma)}{P(\Sigma,Q_L)}\approx \max_{\Sigma\in\mathcal{P}_{L,1}}\frac{W(\Sigma)}{P(\Sigma,Q_L)},
\end{align*}
where $W(\Sigma)$ and $\mathcal{P}_{L,1}$ are defined as in \eqref{e1} and \eqref{M24}, respectively. We observe that this approximation amounts to moving the small-scale cutoff from the field term $W_1$ to the configuration space $\mathcal{P}_{L,1}$.

\medskip

We recall that the field terms $W$ and $W_1$ are naturally coupled through \eqref{M25} and satisfy the property
\begin{align}\label{e17}
W_1(\Sigma)=W(\Sigma)~\mbox{for every $\Sigma\in\mathcal{A}_1$},
\end{align}
where $\mathcal{A}_1$ is the class of lattice animals on the grid of spacing 1. The following result shows that the asymptotic behavior of the two problems is the same.

\begin{lemma}\label{L:6}
There exists a universal positive radius $r$ such that 
\begin{align}\label{e10}
\|\max_{\Sigma\subset Q_L}\frac{W_1(\Sigma)}{P(\Sigma,Q_L)}-\max_{\Sigma\in\mathcal{P}_{L,r}}\frac{W(\Sigma)}{P(\Sigma,Q_L)}\|_2\lesssim\ln^{1/2}L.
\end{align}
\end{lemma}

{\sc Proof.}
As a first step, we show that we can restrict the first maximization problem to polygons. More precisely, for some positive constant $r$,
\begin{align}\label{e12}
\max_{\Sigma\subset Q_L}\frac{W_1(\Sigma)}{P(\Sigma,Q_L)}-\max_{\Sigma\in\mathcal{P}_{L,r}}\frac{W_1(\Sigma)}{P(\Sigma,Q_L)}\lesssim \|\xi_1\|_{L^\infty},
\end{align}
where the r.~h.~s.~ is of the desired order since, in view of Lemma~\ref{L:11} (iii),
\begin{align}\label{M17}
\big\|\|\xi_1\|_{L^\infty}\big\|_2\lesssim\ln^{1/2}L.
\end{align}
Indeed, let $\Sigma_*$ be a maximizer of the first problem on the r.~h.~s.~ of \eqref{e12}. By the same argument as given in \eqref{M7} in Lemma \ref{L:4}, $\Sigma_*$ is a $\kappa$-perimeter minimizer with $\kappa\lesssim 1$ and by Lemma \ref{L:reg} we have
\begin{align}
&\mbox{the connected components of $\partial\Sigma_*\cap Q_L$ are $C^{1,1}$ curves}\nonumber\\&\mbox{with modulus of the curvature $\lesssim1$,}\label{e14}\\
&\partial \Sigma_*\cap B_r(x)\mbox{ is a single curve for all $x\in\partial \Sigma_*\cap Q_L$ for some $r\ll 1$}\label{e11}.
\end{align}
We next construct a polygonal approximation $\Sigma_{*,r}\in\mathcal{P}_{L,r}$ of $\Sigma_*$ such that
\begin{align}
&P(\Sigma_{*,r},Q_L)\leq P(\Sigma_*,Q_L)\quad\mbox{and}\nonumber\\&|W_1(\Sigma_*)-W_1(\Sigma_{*,r})|\lesssim \|\xi_1\|_{L^\infty} P(\Sigma_*,Q_L).\label{e13}
\end{align}
This suffices to get \eqref{e12} by using $\Sigma_{*,r}$ as a competitor for the second problem,
\begin{align*}
\max_{\Sigma\subset Q_L}\frac{W_1(\Sigma)}{P(\Sigma,Q_L)}-\max_{\Sigma\in\mathcal{P}_{L,r}}\frac{W_1(\Sigma)}{P(\Sigma,Q_L)}\leq \frac{|W_1(\Sigma_*)-W_1(\Sigma_{*,r})|}{P(\Sigma_*,Q_L)}\lesssim \|\xi_1\|_{L^\infty}.
\end{align*}

\medskip

We construct $\Sigma_{*,r}$ as follows. For each $C^{1,1}$ curve $\gamma:[0,\len(\gamma)]\to Q_L$ in $\partial \Sigma_*\cap Q_L $, we find $0=t_0<t_1<\ldots t_k\leq t_{k+1}=\len(\gamma)$ such that 
\begin{align*}
&|\gamma(t_i)-\gamma(t_{i+1})|=r\quad\text{for all}\quad i=0,\ldots, k-1;\\
&|\gamma(t_k)-\gamma(t_{k+1})|\in[r,2r].
\end{align*}
By \eqref{e11}, $\{t_i \}_{i=1}^k$ are uniquely defined and linearly interpolating $\{\gamma(t_i)\}_{i=0}^{k+1}$ gives a simple polygonal curve which approximates $\gamma$. Repeating such an approximation on each connected component of $\partial\Sigma_*\cap Q_L$ gives us a family of oriented, simple and disjoint, polygonal curves. Hence they describe the boundary of a set $\Sigma_{*,r}\in\mathcal{P}_{L,r}$. The first inequality in \eqref{e13} is clear by construction. For the second, 
\begin{align*}
|W_1(\Sigma_*)-W_1(\Sigma_{*,r})|&\leq 2\|\xi_1\|_{L^\infty}\lambda(\Sigma_*\Delta\Sigma_{*,r})\\&\leq 2\|\xi_1\|_{L^\infty}\lambda((\partial\Sigma_*)^r)\lesssim \|\xi_1\|_{L^\infty} P(\Sigma_*,Q_L),
\end{align*}
where for the last step we used \eqref{e14}, cf. the proof of \eqref{m5}.

\medskip

Having established \eqref{e12}, to conclude it is enough to show
\begin{align*}
\big\|\max_{\Sigma\in\mathcal{P}_{L,r}}\frac{|W_1(\Sigma)-W(\Sigma)|}{P(\Sigma,Q_L)}\big\|_2\lesssim\ln^{1/2}L.
\end{align*}
Inserting the inner lattice animal approximation $\mathring{\Sigma}$ (defined as \eqref{M16} with $l=1$) and using \eqref{e17}, we have
\begin{align}
|W_1(\Sigma)-W(\Sigma)|&\le|W_1(\Sigma)-W_1(\mathring{\Sigma})|+|W(\Sigma)-W(\mathring{\Sigma})|\nonumber\\&
\le\sum_{i\in\mathcal{I}}|W_1(\Sigma\cap Q_{1,i})|+\sum_{i\in\mathcal{I}}|W(\Sigma\cap Q_{1,i})|\label{M26}
\end{align}
where we recall that $\mathcal{I}:=\{i~|~Q_{1,i}\cap\Sigma\neq\emptyset~\mbox{and}~Q_{1,i}\not\subset\Sigma\}$. We observe that 
\begin{align}\label{M32}
|\mathcal{I}|\lesssim P(\Sigma,Q_L)
\end{align}
since each side of the polygonal boundary $\partial \Sigma$ can intersect at most 3 of the cubes $\{Q_{1,i}\}$. Together with the estimate $|W_1(\Sigma\cap Q_{1,i})|\le2\|\xi_1\|_{L^\infty}$ and  \eqref{M17}, this is enough to control the first r.~h.~s.~term of \eqref{M26}.

\medskip

We now show that the second r.~h.~s.~term of \eqref{M26} satisfies
\begin{align}
\frac{1}{P(\Sigma,Q_L)}\sum_{i\in\mathcal{I}}|W(\Sigma\cap Q_{1,i})|\lesssim\sup_{i=1}^{L^2}\sup_{\text{triangle}~T\subset Q_{1,i}}|W(T)|.\label{M27}
\end{align}
This suffices to conclude: For every $i$ fixed, the set of triangles $T\subset Q_{1,i}$ is finite-dimensional and compact. By standard arguments on the supremum of Gaussians,
\begin{align*}
\big\|\sup_{\text{triangle}~T\subset Q_{1,i}}|W(T)|\big\|_2<\infty.
\end{align*}
The additional supremum in $i=1,\ldots,L^2$ gives a multiplicative contribution of $\ln^{1/2}L$, in view of Lemma~\ref{L:11} $(iii)$. 

\medskip

To prove \eqref{M27}, we introduce 
\begin{align*}
N_i(\Sigma):=|\{\mbox{vertices of the polygon $\Sigma\cap Q_{1,i}$}\}|,
\end{align*}
which allows us to write
\begin{align*}
\sum_{i\in\mathcal{I}}|W(\Sigma\cap Q_{1,i})|\le\sup_{i=1}^{L^2}\frac{|W(\Sigma\cap Q_{1,i})|}{N_i(\Sigma)}\sum_{i\in\mathcal{I}}N_i(\Sigma).
\end{align*}
We note that a vertex of $\Sigma\cap Q_{1,i}$ is either a vertex of $\Sigma$, or of $Q_{1,i}$ or is the intersection of an edge of $\Sigma$ and an edge of $Q_{1,i}$. Therefore, we have
\begin{align*}
\sum_{i\in\mathcal{I}}N_i(\Sigma)\lesssim|\{\mbox{vertices of $\Sigma$}\}|+|\mathcal{I}|+|\{\mbox{edges of $\Sigma$}\}|\overset{\eqref{M32}}{\lesssim}P(\Sigma,Q_L).
\end{align*}
Note that the polygon $\Sigma\cap Q_{1,i}$ can be inductively decomposed into the union of $N_i(\Sigma)-2$ triangles and therefore 
\begin{align*}
\frac{|W(\Sigma\cap Q_{1,i})|}{N_i(\Sigma)}\le\sup_{\text{triangle}~T\subset Q_{1,i}}|W(T)|.
\end{align*}
Combining the last three inequalities, we deduce \eqref{M27} and conclude.
\qed

\subsection{Proof of Theorem~\ref{T:1}.} Let us recall that from \eqref{M1} and \eqref{M3}
\begin{align}\label{M30}
\ln^{1/2}L\ll l\ll\ln^{3/4}L.
\end{align}
Given a threshold $\nu$, we define the good event
\begin{align*}
G:=\big\{\|\rho_l-1\|_{L^\infty}\le\nu l^{-1}\ln^{1/2}L\big\}\quad\mbox{such that}\quad\ln\mathbb{P}(G^c)\le-\ln L,
\end{align*}
where the second inequality holds provided $\nu\gg1$ in view of Lemma~\ref{L:3}. Since $\mathrm{Wass}_\infty(\mu,\lambda)\le\sqrt{2}L$ (by choosing any coupling $\pi$), the contribution of the bad set $G^c$ is negligible:
\begin{align}\label{M31}
\mathbb{E}I(G^c)\mathrm{Wass}_{\infty}(\mu,\lambda)\lesssim Le^{-\ln L}=1.
\end{align}

\medskip

In the good event, we have that $\|\rho_l-1\|_{L^\infty}\ll1$ by \eqref{M30}, which allows us to apply Proposition~\ref{P:1}. By the triangle inequality in \eqref{M2} and \eqref{M30}
\begin{align*}
\mathrm{Wass}_\infty(\mu,\lambda)\approx\max_{\Sigma\subset Q_L}\frac{(\rho_l-1)\lambda(\Sigma)}{P(\Sigma,Q_L)}+O(l).
\end{align*}

It is convenient to record the following concentration result: For any $\Sigma$, 
\begin{align*}
&\frac{W(\Sigma)}{P(\Sigma,Q_L)}\mbox{ is a Gaussian with variance} \\
&\frac{\lambda(\Sigma)\lambda(Q_L\backslash \Sigma)}{\lambda(Q_L)}\cdot\frac{1}{P(\Sigma,Q_L)^2}\leq \frac{\min\{\lambda(\Sigma),\lambda(Q_L\backslash \Sigma)\}}{P(\Sigma,Q_L)^2}\lesssim 1
\end{align*}
where the last inequality is the relative isoperimetric inequality. For the supremum of Gaussians with bounded variance,  Borell's inequality gives (cf.~\cite{BLM}) 
\begin{align*}
\|\max_{\Sigma\in\mathcal{P}_{L,1}}\frac{W(\Sigma)}{P(\Sigma,Q_L)}-\mathbb{E}\max_{\Sigma\in\mathcal{P}_{L,1}}\frac{W(\Sigma)}{P(\Sigma,Q_L)}\|_2\lesssim1.
\end{align*}

Combining it with Lemma~\ref{L:4} and Lemma~\ref{L:6}, we get
\begin{align*}
\big\|\max_{\Sigma\subset Q_L}\frac{(\rho_l-1)\lambda(\Sigma)}{P(\Sigma,Q_L)}-\mathbb{E}\max_{\Sigma\in\mathcal{P}_{L,lr}}\frac{W(\Sigma)}{P(\Sigma,Q_L)}\big\|_1\lesssim\ln^{1/2}L.
\end{align*}
Hence in the good event $G$ and for some random variable $X$ with $\|X\|_1\lesssim1$
\begin{align*}
\mathrm{Wass}_\infty(\mu,\lambda)\approx\mathbb{E}\max_{\Sigma\in\mathcal{P}_{L,lr}}\frac{W(\Sigma)}{P(\Sigma,Q_L)}+X\ln^{1/2}L+O(l).
\end{align*}
Taking expected values and using that $\mathbb{P}(G^c)\ll1$,
\begin{align*}
\mathbb{E}\mathrm{Wass}_\infty(\mu,\lambda)\overset{\eqref{M31}}&{=}\mathbb{E}I(G)\mathrm{Wass}_{\infty}(\mu,\lambda)+O(1)\\&\approx\mathbb{E}\max_{\Sigma\in\mathcal{P}_{L,lr}}\frac{W(\Sigma)}{P(\Sigma,Q_L)}+O(\ln^{1/2}L)+O(l).
\end{align*}
Since $\ln(L/lr)\approx\ln L$, by the rescaling $x=lr\hat x$ and applying Theorem~\ref{T:4} we conclude.

\subsection{Bipartite Matching}
In this subsection, we turn to the proof of Theorem \ref{T:3}. Since it is almost analogous, we just sketch it and comment on the differences. As before, we coarse-grain the measures $\mu_X,\mu_Y$ to 
\begin{align*}
    \rho_l\lambda:=\sum_{i=1}^{(L/l)^2}\frac{\mu_X(Q_{l,i})}{\lambda(Q_{l,i})}\lambda\llcorner Q_{l,i}\quad\text{and}\quad\sigma_l\lambda:=\sum_{i=1}^{(L/l)^2}\frac{\mu_Y(Q_{l,i})}{\lambda(Q_{l,i})}\lambda\llcorner Q_{l,i}.
\end{align*}
Again by triangle inequality,
\begin{align}
&|\mathrm{Wass}_\infty(\mu_X,\mu_Y)-\mathrm{Wass}_\infty(\rho_l\lambda,\sigma_l\lambda)|\nonumber\\&\leq \mathrm{Wass}_\infty(\mu_X,\rho_l\lambda)+\mathrm{Wass}_\infty(\mu_Y,\sigma_l\lambda)\lesssim l.\label{a12}
\end{align}
Thus it is enough to treat $\mathrm{Wass}_\infty(\rho_l\lambda,\sigma_l\lambda)$, as long as we choose $l\ll\ln^{3/4}L$. The natural extension of Proposition \ref{P:1} is the following statement.
\begin{proposition}\label{P:2}
    Let $\rho,\sigma$ be densities satisfying $\int\rho d\lambda=\int\sigma d\lambda$. Whenever $\|\rho-1\|_{L^\infty},\|\sigma-1\|_{L^\infty}\ll1$,
    \begin{align*}
        \mathrm{Wass}_\infty(\rho\lambda,\sigma\lambda)=\big(1+O(\|\rho-1\|_{L^\infty}+\|\sigma-1\|_{L^\infty})\big)\max_{\Sigma\subset Q_L}\frac{(\rho-\sigma)\lambda(\Sigma)}{P(\Sigma,Q_L)}.
    \end{align*}
\end{proposition}
{\sc Proof.}
We begin by showing the analogue of Lemma \ref{L:1}, namely
\begin{align}\label{a13}
    &\mathrm{Wass}_\infty(\rho\lambda,\sigma\lambda)\nonumber\\&\leq(1-\|\sigma-1\|_{L^\infty}-\|\rho-\sigma\|_{L^\infty})^{-1}\max_{\Sigma\subset Q_L}\frac{(\rho-\sigma)\lambda(\Sigma)}{P(\Sigma,Q_L)}.
\end{align}
Let us denote the r.~h.~s.~by $R$. Given $\Sigma\subset Q_L$ and $r$ as in (\ref{m2}), we have
\begin{align*}
    \sigma\lambda(\Sigma^R)-\sigma\lambda(\Sigma)&=\int_{Q_L}\big(I({\Sigma^R})-I(\Sigma)\big)\sigma(x)dx\\
    &\geq (1-\|1-\sigma\|_{L^\infty})(\lambda(\Sigma^R)-\lambda(\Sigma))\\
    &\geq (1-\|1-\sigma\|_{L^\infty})RP(\Sigma^r,Q_L)
\end{align*}
Proceeding as in the proof of Lemma \ref{L:1}, we recover the condition
\begin{align*}
RP(\Sigma^r,Q_L)(1-\|\sigma-1\|_{L^\infty}-\|\rho-\sigma\|_{L^\infty})-(\rho-\sigma)\lambda(\Sigma^r)\ge0,
\end{align*}
which holds by definition of $R$.

\medskip

The needed refinement of Lemma \ref{L:2} is
\begin{align*}
    \mathrm{Wass}_\infty(\rho\lambda,\sigma\lambda)\ge((1+\|1-\sigma\|_{L^\infty})\big(1+\|\rho-\sigma\|_{L^\infty})\big)^{-1}\max_{\Sigma\subset Q_L}\frac{(\rho-\sigma)\lambda(\Sigma)}{P(\Sigma,Q_L)}.
\end{align*}
The proof is the same up to choosing $\kappa=\|\rho-\sigma\|_{L^\infty}/R_*$ and extending (\ref{m5}) to
\begin{align*}
\sigma\lambda(\Sigma_*^R)-\sigma\lambda(\Sigma_*)\leq(1+\|1-\sigma\|_{L^\infty})(1+\kappa R)RP(\Sigma_*,Q_L).
\end{align*}
This however follows from (\ref{a15}) and 
\begin{align*}
\sigma\lambda(\Sigma^R_*)-\sigma\lambda(\Sigma_*)\leq(1+\|1-\sigma\|_{L^\infty})(\lambda(\Sigma_*^R)-\lambda(\Sigma_*)).
\end{align*}
As before this proves 
\begin{align*}
    \sigma\lambda(\Sigma_*^R)<\rho\lambda(\Sigma_*)\quad\mbox{if}\quad R<((1+\|1-\sigma\|_{L^\infty})(1+\|\rho-\sigma\|_{L^\infty}))^{-1}R_*.
\end{align*}
Absorbing $\|1-\sigma\|_{L^\infty}\cdot\|\rho-\sigma\|_{L^\infty}$ into one of the other terms we recover Proposition \ref{P:2}.
\qed

\medskip

Combining Proposition \ref{P:2} and Lemma \ref{L:3}, with overwhelming probability
\begin{align*}
    \mathrm{Wass}_\infty(\mu_X,\mu_Y)\approx\max_{\Sigma\subset Q_L}\frac{(\rho_l-\sigma_l)\lambda(\Sigma)}{P(\Sigma,Q_L)}\quad\mbox{provided}\quad \ln^{1/2}L\ll l\ll\ln^{3/4}L.
\end{align*}
To approximate $(\rho_l-\sigma_l)\lambda$ by a Gaussian term, we construct two independent fields $W_l^1$ and $W_l^2$ with the same law as $W_l$ in \eqref{M4} and satisfying
\begin{align*}
&(\rho_l-1)\lambda\approx W_l^1\quad\mbox{and}\quad(\sigma_l-1)\lambda\approx W_l^2\\
&\mbox{so that}\quad(\rho_l-\sigma_l)\lambda\approx \tilde{W}_l:=W_l^1-W_l^2.
\end{align*}
Lemma \ref{L:5} implies, after application of the triangle inequality, that
\begin{align*}
\|\max_{\Sigma\in\mathcal{A}_l}\frac{|(\rho_l-\sigma_l)\lambda(\Sigma)-\tilde{W}_l(\Sigma)|}{P(\Sigma,Q_L)}\|_1\lesssim l^{-1}\ln L.
\end{align*}
All the other results translate directly to this case, which, together with $\tilde{W}_l=_{\text{law}}\sqrt2W_l$, proves Theorem \ref{T:3}. 

\section{Acknowledgements}

AW acknwoledges funding by the DFG Priority Programm ``Random Geometric Systems'' (SPP 2265) - project P10 ``Optimal matching and balancing transport''. This research has been supported by the DFG through the Research Training Group ``Rigorous Analysis of Complex Random Systems'' (RTG 3027) funded by the Deutsche Forschungsgemeinschaft (DFG, German Research Foundation – Project Number 524444762) and under Germany's Excellence Strategy EXC 2044/2 –390685587, Mathematics Münster: Dynamics–Geometry–Structure.
\section{Appendix}

\begin{lemma}[{\cite[Lemma 5 and Lemma 14]{OPW}}]\label{L:11}
For $s\in[1,\infty)$ and positive random variables $X, \{X_i\}_{i=1}^\infty$
\begin{align*}
&(i)&&\ln\mathbb{P}(X\ge \nu)\lesssim-(\nu/\|X\|_s)^s\quad\mbox{for
all}~\nu\gg\|X\|_s;\\
&(ii)&&\mbox{if}\quad\ln\mathbb{P}(X\ge\nu)\le-\nu^s\quad\mbox{for}~\nu\ge\nu_0>0\quad\mbox{then}
\quad\|X\|_s-\nu_0\lesssim1;\\
&(iii)&&\|\max_{i=1}^n X_i\|_s\lesssim(\ln^{1/s}n)\max_{i=1}^n\|X_i\|_s.
\end{align*}
\end{lemma}

\begin{lemma}\label{L:reg}
Assume that $\Sigma\subset Q_L$ satisfies for every ball $B$,
\begin{align}\label{m8}
P(\Sigma,B\cap Q_L)\le P(\Sigma',B\cap Q_L)+\kappa|\Sigma\Delta\Sigma'|\quad\mbox{provided $\Sigma\Delta\Sigma'\subset B$}.
\end{align}
Then $\partial\Sigma\cap Q_L$ is the union of $C^{1,1}$ curves with curvature $\le\kappa$. Moreover, for every $R\ll\kappa^{-1}$ and $x\in\partial\Sigma\cap Q_L$,
\begin{align}\label{m9}
\partial\Sigma\cap B_R(x)~\mbox{consists of a single curve}.
\end{align}
\end{lemma}

{\sc Proof.} We note that, upon reflection along the edges of $Q_L$, one may assume that $\Sigma$ satisfies \eqref{m8} in the entire plane $\mathbb{R}^2$, hence one may forget about the square $Q_L$. Sets $\Sigma$ satisfying \eqref{m8} are said to be $\kappa$-perimeter minimizers \cite[Equation 21.2]{Mag}. It is a classical result that such sets satisfy a $C^{1,\gamma}$ partial regularity for all $\gamma<1/2$ and that \eqref{m9} holds at regular points \cite[Theorem 26.3]{Mag}. This can be upgraded to an everywhere $C^{1,\gamma}$ regularity in dimension $d\le 7$ (see \cite[Theorems 26.5 and 28.1]{Mag}); $d=2$ in our case. 

\medskip

It is a classical result that in dimension $d=2$, the $C^{1,\gamma}$ regularity can be improved to $C^{1,1}$. This is done by writing the Euler-Lagrange equations and showing that the curvature is $\le\kappa$. For the convenience of the reader, we display here a proof. The boundary $\partial\Sigma$ can be parameterized locally around each point by the graph of a $C^{1,\gamma}$ function $u:(-a,a)\to\mathbb{R}$. Given $\phi\in C_c^\infty(-a,a)$ and a real $t$ small enough, we test the condition \eqref{m8} with $\Sigma'$ whose boundary is locally given by the function $u+t\phi$ to get
\begin{align*}
\int_{-a}^a\sqrt{1+(u')^2}\le\int_{-a}^a\sqrt{1+(u'+t\phi')^2}+\kappa\big|\int_{-a}^a|u+t\phi|-\int_{-a}^a|u|\big|.
\end{align*}
Dividing by $t$ and taking the limit $t\to0$, we learn that
\begin{align*}
-\int f'(u')\phi'\le\kappa\int|\phi|\quad\mbox{where}\quad f(s):=\sqrt{1+s^2}.
\end{align*}

\medskip

Let us now deduce from this that the $C^{0,\gamma}$ function $f'(u')$ is indeed $\kappa$-Lipschitz, hence Lebesgue-a.~e.~differentiable, with derivative $\le\kappa$. Since $(f'(u'(x)))'$ is indeed the curvature of $\partial\Sigma$ at the point $(x,u(x))$, this is enough to conclude. To this end, let us choose $\phi$ so that
\begin{align*}
\phi'(x)=\rho_\varepsilon(x-x_1)-\rho_\varepsilon(x-x_0)
\end{align*}
where $\rho_\varepsilon$ is a mollification of the $\delta$-Dirac at scale $\varepsilon$. Testing with $\pm\phi$ and letting $\varepsilon$ go to 0, we deduce the desired Lipschitz continuity
\begin{align*}
|f'(u'(x_1))-f'(u'(x_0))|\le\kappa|x_0-x_1|.
\end{align*}
\qed

\medskip
{\sc Proof of Corollary \ref{L:7}.} 
We note that it is enough to prove 
\begin{align*}
        &\ln\mathbb P\big(\sup_{t\leq T}|\mu^{1d}([0,t])-\frac{n}{T}t-\sqrt{\frac nT}\big(\xi^{1d}([0,t])-\frac{t}{T}\xi^{1d}([0,T])\big)|\ge\nu+C\ln n\big)\\
        &\lesssim-\nu,
    \end{align*}
since the statement then follows by triangle inequality. The above inequality however is consequence of a simple scaling argument. Indeed,
\begin{align*}
    \xi^{1d}(\cdot)=T^{1/2}\hat{\xi}^{1d}(\cdot/T)
\end{align*}
is a white noise on the interval $[0,1]$. Let $\hat{\mu}^{1d}=\sum_{i=1}^n\delta_{\hat{Z}_i}$ be the random measure from Remark \ref{rk:5} coupled to $\hat\xi^{1d}$. Dilating the points $\hat{Z}_i$ to the interval $[0,T]$ by $Z_i:=T\hat{Z}_i$ defines the measure $\mu^{1d}:=\sum_{i=1}^n\delta_{Z_i}$, which is a binomial process on the interval $[0,T]$. Setting $t=T\hat t$,
\begin{align*}
    |\mu^{1d}([0,t])-\frac{n}{T}t-&\sqrt{\frac nT}\big(\xi^{1d}([0,t])-\frac{t}{T}\xi^{1d}([0,T])\big)|\\&=|\hat\mu^{1d}([0,\hat {t}])-n\hat t-\sqrt{n}\big(\hat\xi^{1d}([0,\hat t])-\hat t\hat\xi^{1d}([0,1])\big)|.
\end{align*}
Taking the supremum over all $t\leq T$ amounts to the supremum over all $\hat t\leq 1$, so that the statement follows from the definition of $\hat\mu^{1d}$.
\qed


\begin{thebibliography}{alpha} 

\bibitem{AKT}
    \newblock Mikl{\'o}s Ajtai, J{\'a}nos Koml{\'o}s, G{\'a}bor Tusn{\'a}dy.
    \newblock On optimal matchings.
    \newblock \emph{Combinatorica}, {\bf 4}, 4, 259--264 (1984), Springer-Verlag Berlin/Heidelberg.

\bibitem{AST}
    \newblock Luigi Ambrosio, Federico Stra, Dario Trevisan.
    \newblock A PDE approach to a 2-dimensional matching problem.
    \newblock \emph{Probability Theory and Related Fields}, {\bf 173}, 1, 433--477 (2019), Springer.

\bibitem{ATV}
    \newblock Luigi Ambrosio, Dario Trevisan, Federico Vitillaro.
    \newblock Sharp PDE estimates for random two-dimensional bipartite matching with power cost function.
    \newblock \emph{Rendiconti Lincei}, {\bf 35}, 2, 323--342 (2024).

\bibitem{BLM}
    \newblock St\'ephane Boucheron, G\'abor Lugosi, Pascal Massart.
    \newblock \emph{Concentration inequalities}. Oxford University Press, (2013).

\bibitem{CLPS}
    \newblock Sergio Caracciolo, Carlo Lucibello, Giorgio Parisi, Gabriele Sicuro.
    \newblock Scaling hypothesis for the Euclidean bipartite matching problem.
    \newblock \emph{Physical Review E}, {\bf 90}, 1, 012118 (2014), APS.
    
\bibitem{COP}
    \newblock Xiaopeng Cheng, Felix Otto and Matteo Palmieri.
    \newblock Asymptotics of a planar isoperimetric problem with a white-noise volume term.
    \newblock \emph{arXiv preprint}, 2026.

\bibitem{EG}
    \newblock Lawrence C. Evans, Ronald F. Gariepy.
    \newblock \emph{Measure theory and fine properties of functions}. Chapman \& Hall, (2015).

\bibitem{GHO}
    \newblock Michael Goldman, Martin Huesmann, Felix Otto.
    \newblock Almost sharp rates of convergence for the average cost and displacement in the optimal matching problem.
    \newblock \emph{The Abel Symposium}, 93--103 (2023), Springer.

\bibitem{KLM}
    \newblock Richard Manning Karp, Michael Luby, and Alberto Marchetti-Spaccamela.
    \newblock A probabilistic analysis of multidimensional bin packing problems
    \newblock \emph{Proceedings of the sixteenth annual ACM symposium on Theory of computing} 289--298 (1984).

\bibitem{KMT}
    \newblock János Komlós, Péter Major, and Gábor Tusnády.
    \newblock An approximation of partial sums of independent RV'-s, and the sample DF.~I.
    \newblock \emph{Zeitschrift für Wahrscheinlichkeitstheorie und verwandte Gebiete} \textbf{32(1)},  111--131 (1975).

\bibitem{LS}
	\newblock Thomas Leighton and Peter Shor.
	\newblock Tight bounds for minimax grid matching with applications to the average case analysis of algorithms.
	\newblock \emph{Combinatorica} \textbf{9}, 161–187 (1989).

\bibitem{Mag}
    \newblock Francesco Maggi.
    \newblock \emph{Sets of finite perimeter and geometric variational problems: an introduction to Geometric Measure Theory.}
    \newblock  Cambridge University Press \textbf{135} (2012).

\bibitem{Mas}
    \newblock Pascal Massart.
    \newblock Strong approximation for multivariate empirical and related processes, via KMT constructions.
    \newblock \emph{The Annals of Probability}, 266--291 (1989).

\bibitem{OP}
	\newblock Felix Otto, and Matteo Palmieri.
    \newblock Quantitative homogenization of the maximal action of curves in a Brownian potential.
	\newblock \emph{arXiv preprint} 2604.27519, 2026.

\bibitem{OPW}
	\newblock Felix Otto, Matteo Palmieri, and Christian Wagner.
	\newblock On minimizing curves in a Brownian potential.
	\newblock \emph{Probability Theory and Related Fields}, 1--62 (2026), Springer.

\bibitem{Pol}
    \newblock David Pollard
    \newblock {\emph{A User's Guide to Measure Theoretic Probability}.}
    \newblock Cambridge Series in Statistical and Probabilistic Mathematics \textbf{8}, (2011).

\bibitem{RW}
	\newblock Tobias Ried, and Christian Wagner.
	\newblock Large scale regularity and correlation length for almost length-minimizing random curves in the plane.
	\newblock \emph{arXiv preprint} {	arXiv:2412.17625}, 2026.

\bibitem{T2}
	\newblock Michel Talagrand.
	\newblock {{Matching theorems and discrepancy computations using majorizing measures}}.
	\newblock \emph{Journal of the American Mathematical Society} \textbf{7}, 455-537 (1994).

\bibitem{T}
	\newblock Michel Talagrand.
	\newblock {\emph{Upper and lower bounds for stochastic processes:
	decomposition theorems}}.
	\newblock Vol. \textbf{60}.
	\newblock Springer Nature, (2022).
    
\end{thebibliography}
\end{document}